\documentclass[11pt,reqno]{amsart}

\usepackage{amsfonts, amsthm, amsmath}
\allowdisplaybreaks[4]

\usepackage{rotating}

\usepackage{tikz}

\usepackage{graphics}

\usepackage{amssymb}

\usepackage{amscd}

\usepackage[latin2]{inputenc}

\usepackage{t1enc}

\usepackage[mathscr]{eucal}

\usepackage{ferrers}

\usepackage{indentfirst}

\usepackage{graphicx}

\usepackage{graphics}

\usepackage{pict2e}

\usepackage{mathrsfs}

\usepackage{enumerate}
\usepackage[pagebackref]{hyperref}
\hypersetup{colorlinks=true}
\usepackage{cite}
\usepackage{color}
\usepackage{epic}
\usepackage{hyperref} 
\usepackage{framed}
\usepackage{mathabx}

\numberwithin{equation}{section}
\def\red{\textcolor{red}}
\theoremstyle{plain}

\newtheorem{theorem}{Theorem}[section]

\newtheorem{lemma}[theorem]{Lemma}

\theoremstyle{definition}

\newtheorem{Def}[theorem]{Definition}

\newtheorem{example}[theorem]{Example}

\newtheorem{remark}[theorem]{Remark}

\newtheorem{?}[theorem]{Problem}

\newcommand{\PP}{\mathcal{P}}
\newcommand{\F}{\mathcal{F}}
\newcommand{\D}{\mathcal{D}}
\newcommand{\OO}{\mathcal{O}}
\newcommand{\sbf}{\mathbf{s}}
\newcommand{\lbf}{\mathbf{l}}
\newcommand{\vbf}{\mathbf{v}}

\newcommand{\LL}{\mathcal{L}}
\newcommand{\redd}{\textcolor{red}}
\def\asc{\mathrm{asc}}
\def\fb{\mathrm{fb}}
\def\lb{\mathrm{lb}}
\newcommand\qbin[2]{{\left[\begin{matrix} #1 \\ #2 \end{matrix} \right]}}

\begin{document}

\title[$m$-falling lecture hall theorem]{A lecture hall theorem for $m$-falling partitions}

\author[S. Fu]{Shishuo Fu}
\address[Shishuo Fu]{College of Mathematics and Statistics, Chongqing University, Huxi campus LD506, Chongqing 401331, P.R. China}
\email{fsshuo@cqu.edu.cn}

\author[D. Tang]{Dazhao Tang}

\address[Dazhao Tang]{College of Mathematics and Statistics, Chongqing University, Huxi campus LD206, Chongqing 401331, P.R. China}
\email{dazhaotang@sina.com}

\author[A. J. Yee]{Ae Ja Yee}
\address[Ae Ja Yee]{Department of Mathematics, The Pennsylvania State University, University Park, PA 16802, USA}
\email{yee@psu.edu}

\date{\today}

\dedicatory{To our mentor and friend, George E. Andrews, on his 80th birthday.}

\begin{abstract}
For an integer $m\ge 2$, a partition $\lambda=(\lambda_1,\lambda_2,\ldots)$ is called $m$-falling, a notion introduced by Keith, if the least nonnegative residues mod $m$ of $\lambda_i$'s form a nonincreasing sequence. We extend a bijection originally due to the third author to deduce a lecture hall theorem for such $m$-falling partitions. A special case of this result gives rise to a finite version of Pak-Postnikov's $(m,c)$-generalization of Euler's theorem. Our work is partially motivated by a recent extension of Euler's theorem for all moduli, due to Keith and Xiong. We note that their result actually can be refined with one more parameter.
\end{abstract}

\subjclass[2010]{05A17, 11P83}

\keywords{Partitions; Stockhofe-Keith map; lecture hall partitions; $m$-falling partitions.}

\maketitle


\section{Introduction}\label{sec1: intro}

A \emph{partition} $\lambda$ of a positive integer $n$ is a nonincreasing sequence of positive integers $(\lambda_1,\lambda_2,\ldots,\lambda_r)$ such that $\sum_{i=1}^{r}\lambda_i=n$. The $\lambda_i$'s are called the \emph{parts} of $\lambda$, and $n$ is called the \emph{weight} of $\lambda$, usually denoted as $|\lambda|$. For convenience, we often allow parts of size zero and append as many zeros as needed. 

Being widely perceived as the genesis of the theory of partition, Euler's theorem asserts that the set of partitions of $n$ into odd parts and the set of partitions of $n$ into distinct parts are equinumerous. Equivalently,
\begin{equation*}
\prod_{i=1}^{\infty} (1+q^i)=\prod_{i=1}^{\infty} \frac{1}{1-q^{2i-1}}.
\end{equation*}

Among numerous generalizations and refinements of Euler's theorem \cite{And1, And2, BME1, BME2, BME3, CGJL, CYZ, Fin, Hir, Moo, RV, PP, Pak, Sav, Sta, Syl, Yee1}, the one that arguably attracted the most attention is the following finite version named the \emph{Lecture Hall Theorem}, first discovered by Bousquet-M\'elou and Eriksson.

If $\lambda=(\lambda_1, \ldots, \lambda_n)$ is a partition of length $n$ with some parts possibly zero such that
\begin{equation}
 \frac{\lambda_1}{n}\ge\frac{\lambda_2}{n-1}\ge\cdots\ge\frac{\lambda_n}{1}\ge 0, \label{LHC}
\end{equation}
then $\lambda$ is called a \emph{lecture hall partition of length $n$}. Let $\LL_n$ be the set of lecture hall partitions of length $n$.

\begin{theorem}[Theorem 1.1, \cite{BME1}]
For $n\ge 1$,
\begin{align}\label{id:lecture-hall}
\sum_{\lambda\in\LL_n}q^{|\lambda|}=\prod_{i=1}^{n}\frac{1}{1-q^{2i-1}}.
\end{align}
\end{theorem}

It can be easily checked that any partition $\lambda$ into distinct parts less than or equal to $n$ satisfies the inequality condition in \eqref{LHC}. That is,  $\lambda\in\LL_n$ for any $n\ge \lambda_1$, which shows that \eqref{id:lecture-hall} indeed yields Euler's theorem when $n\rightarrow \infty$.

In 1883, Glaisher \cite{Gla} found a purely bijective proof of Euler's theorem and was able to extend it to the equinumerous relationship between partitions with parts repeated less than $m$ times and partitions into non-multiples of $m$ for any $m\ge 2$. 
That is,
\begin{equation*}
\prod_{n=1}^{\infty}\left(1+q^{n}+\cdots +q^{(m-1)n}\right)=\prod_{i=1\atop i\not\equiv 0 \pmod{m}}^{\infty} \frac{1}{(1-q^{i})}.
\end{equation*}

Recently, Xiong and Keith \cite{KX} obtained a substantial refinement of Glaisher's result with respect to certain partition statistics, which we define next.

Throughout this paper, we will assume that $m\ge 2$.  For any partition $\lambda=(\lambda_1,\lambda_2,\ldots)$, let
$$s_i(\lambda)=\lambda_i-\lambda_{i+1}+\lambda_{m+i}-\lambda_{m+i+1}+\lambda_{2m+i}-\lambda_{2m+i+1}+\cdots, \quad 1\le i\le m.$$
We define its \emph{$m$-alternating sum type} to be the $(m-1)$-tuple $\sbf(\lambda):=(s_1(\lambda),\ldots,s_{m-1}(\lambda))$ and its \emph{$m$-alternating sum} $s(\lambda):=\sum_{i=1}^{m-1} s_i(\lambda)$.
We  note that the $m$-alternating sum type of $\lambda$ does not put any restriction on $s_m$. 

Similarly, let
$$
\ell_i(\lambda)=\#\{j:\lambda_j\equiv i \pmod m\}, \quad 1\le i\le m.
$$
We define its \emph{$m$-length type} to be the $(m-1)$-tuple $\lbf(\lambda):=(\ell_1(\lambda),\ell_2(\lambda),\ldots,\ell_{m-1}(\lambda))$ and its \emph{$m$-length} $\ell(\lambda)=\sum_{i=1}^{m-1}\ell_i(\lambda)$.
Note that the $m$-length type of $\lambda$ is independent of the parts in $\lambda$ that are multiples of $m$. 

Let us define the following two subsets of partitions:
\begin{itemize}
		\item $\D_m$: the set of partitions in which each non-zero part can be repeated at most $m-1$ times;
		\item $\OO_m$: the set of partitions in which each non-zero part is not divisible by $m$, called \emph{$m$-regular} partitions.
\end{itemize}

\begin{theorem}[Theorem 2.1, \cite{KX}]\label{KXthm}
For $m\ge 2$,
\begin{equation*}
\sum_{\mu \in \D_m} z_1^{s_1(\mu)} \cdots z_{m-1}^{s_{m-1}(\mu)} q^{|\mu |} =\sum_{\lambda \in \OO_m} z_1^{\ell_1(\lambda)} \cdots z_{m-1}^{\ell_{m-1}(\lambda)}q^{|\lambda|}.
\end{equation*}
\end{theorem}

The natural desire to find certain ``lecture hall version'' for the result of Xiong and Keith motivated us to take on this investigation. While the version with full generality matching their result is yet to be found, we do obtain a lecture hall theorem for $m$-falling partitions.

A partition $\lambda=(\lambda_1,\lambda_2,\ldots)$ is called $m$-falling, which was introduced by Keith in \cite{Kei-the}, if the least nonnegative residues mod $m$ of $\lambda_i$'s form a nonincreasing sequence.  We denote the set of $m$-falling and $m$-regular partitions ($m$-falling regular partitions for short) as $\OO_{m\searrow}$.
For $n\ge 1$, let
\begin{align*}
\OO_{m\searrow}^n&:=\{\lambda\in\OO_{m\searrow}:\lambda_1<nm\}
\end{align*}
and $\LL_{m\searrow}^n$ be a subset of $\D_m$ with certain ratio conditions between parts. Due to the complexity of the conditions, the definition of $\LL_{m\searrow}^n$ is postponed to section~\ref{sec3: m-falling}.  A partition in $\LL_{m\searrow}^n$ is called an $m$-falling lecture hall partition.

We now state the main result of this paper.

\begin{theorem}[$m$-falling lecture hall theorem]\label{thm:mfallingLH}
For $m\ge 2$ and $n\ge 1$,
\begin{align}
\sum_{\mu \in \LL_{m\searrow}^n}z_1^{s_1(\mu)}\cdots z_{m-1}^{s_{m-1}(\mu)}q^{|\mu|}=\sum_{\lambda \in \OO_{m\searrow}^n} z_1^{\ell_1(\lambda)}\cdots z_{m-1}^{\ell_{m-1}(\lambda)}q^{|\lambda|}.
\end{align}
\end{theorem}


Another result of this paper is a refinement of Theorem~\ref{KXthm}. Let us consider the \emph{residue sequence} of a partition. Namely, for $\lambda=(\lambda_1,\lambda_2,\ldots)$, we take for each part the least non-negative residue modulo $m$ and denote the resulting sequence as $v(\lambda)=v_1v_2\cdots$. Recall the permutation statistic \emph{ascent}:
\begin{align*}
\asc(w)=\#\{i:1\le i <n,\;w_i<w_{i+1}\},
\end{align*}
for any word $w=w_1\cdots w_n$, which is consisted of totally ordered letters. We extend this statistic to partitions via their residue sequences and let $\asc(\lambda)=\asc(v(\lambda))$. 

We have the following refinement of  Theorem \ref{KXthm}.
\begin{theorem}\label{thm: one para refine}
For $m\ge 2$,
\begin{align}
\sum_{\mu \in\D_m}z_1^{s_1(\mu)}\cdots z_{m-1}^{s_{m-1}(\mu)}z_m^{s_m(\mu)}q^{|\mu|}=\sum_{\lambda\in\OO_m}z_1^{\ell_1(\lambda)}\cdots z_{m-1}^{\ell_{m-1}(\lambda)}z_m^{\left\lfloor\frac{\lambda_1}{m}\right\rfloor-\asc(\lambda)}q^{|\lambda|}.
\end{align}
\end{theorem}


To make this paper self-contained, in the next section we first recall the Stockhofe-Keith map and then prove Theorem~\ref{thm: one para refine}. In section~\ref{sec3: m-falling}, we define $m$-falling lecture hall partitions and prove Theorem~\ref{thm:mfallingLH}, one special case of which gives rise to a lecture hall theorem (see Theorem~\ref{thm: mc-lecture hall}) for Pak-Postnikov's $(m,c)$-generalization \cite{PP} of Euler's theorem. We conclude in the final section with some outlook for future work.

\section{Preliminaries and a proof of Theorem~\ref{thm: one para refine}}\label{sec2: one para refine}

In this section, we first recall further definitions and notions involving partitions for later use. After that, we will recap the Stockhofe-Keith map and prove Theorem~\ref{thm: one para refine}.

Given two (infinite) sequences $\lambda=(\lambda_1,\lambda_2,\ldots)$ and $\mu=(\mu_1,\mu_2,\ldots)$,  we define the usual linear combination $k\lambda+l\mu$ as
$$
k\lambda+l\mu=(k\lambda_1+l\mu_1,k\lambda_2+l\mu_2,\ldots)
$$
for any two nonnegative integers $k$ and $l$.

For a partition $\lambda=(\lambda_1,\ldots,\lambda_r)$, its \emph{conjugate partition} $\lambda'=(\lambda_1',\ldots,\lambda_s')$ is a partition resulting from choosing $\lambda_i'$ as the number of parts of $\lambda$ that are not less than $i$  \cite[Definition~1.8]{Andr1976}.

The following lemma (see for instance \cite[Lemma 1]{KX}) follows  via the conjugation of partitions.

\begin{lemma}
The conjugation map $\lambda\mapsto \lambda'$ is a weight-preserving bijection such that
\begin{enumerate}[1)]
	\item $\sbf(\lambda)=\lbf(\lambda')$,
	\item $\lambda_1-s(\lambda)=\ell_m(\lambda')$.
\end{enumerate}
\end{lemma}

\proof
1) This immediately follows via the conjugation of partitions, so we omit the details.

2) Again, by conjugation, we see that $s_m(\lambda)=\ell_m(\lambda')$. Also, by the definition,
\begin{equation*}
\lambda_1=s_1(\lambda)+\cdots +s_{m}(\lambda)=s(\lambda)+s_m(\lambda).
\end{equation*}
Thus $\lambda_1-s(\lambda)=\ell_m(\lambda')$.
\endproof

Using conjugation, we can derive an interesting set of partitions that are equinumerous to $\D_m$, namely $m$-flat partitions:
\begin{itemize}
	\item $\F_m$: the set of partitions in which the differences between consecutive parts are at most $m-1$, called \emph{$m$-flat} partitions.
\end{itemize}

\begin{remark} \label{rem0}
The two sets $\D_m$ and $\F_m$ are clearly in one-to-one correspondence via conjugation.
\end{remark}

\subsection{Stockhofe-Keith map \texorpdfstring{$\phi: \OO_m\to \D_m$}{}}

Given any partition $\lambda$, we define its \emph{base $m$-flat partition}, denoted as $\beta(\lambda)$, as follows. Whenever there are two consecutive parts $\lambda_i$ and $\lambda_{i+1}$ with $\lambda_{i}-\lambda_{i+1}\ge m$, we subtract $m$ from each of the parts $\lambda_1,\lambda_2,\ldots, \lambda_i$. We repeat this until we reach a partition in $\F_m$, which is taken to be $\beta(\lambda)$.

Suppose we are given a partition $\lambda\in\OO_m$. We now describe step-by-step how to get a partition $\phi(\lambda)=\mu\in\D_m$ via the aforementioned Stockhofe-Keith map $\phi$.
\begin{enumerate}
	\item[Step 1:]  Decompose $\lambda=m\sigma+\beta(\lambda)$.
	\item[Step 2:] Insert each part in $m\sigma'$, from the largest one to the smallest one, into $\beta(\lambda)$ according to the following insertion method. Note that after each insertion, we always arrive at a new $m$-flat partition. In particular, the final partition we get, say $\tau$, is in $\F_m$ as well.
	\item[Step 3:] Conjugate $\tau$ to get $\mu=\tau'\in\D_m$.
\end{enumerate}

\begin{framed}
\begin{center}
 {\bf Insertion method to get $\tau\in\F_m$}
\end{center}
\noindent Initiate $\tau=(\tau_1,\tau_2,\ldots)=\beta(\lambda)$. Note that parts in $m\sigma'$ are necessarily multiples of $m$. Suppose we currently want to insert a part $km$ into $\tau$.

If $km-\tau_1\ge m$, then find the unique integer $i$, $1\le i\le k$, such that $$(\tau_1+m,\tau_2+m,\ldots,\tau_i+m, (k-i)m, \tau_{i+1},\ldots)$$ is still a partition in $\F_m$. Replace $\tau$ with this new partition.

Otherwise, we simply insert $km$ into $\beta(\lambda)$ as a new part and replace $\tau$ with this new partition.
\end{framed}

For example, let us take $m=3$ and $\lambda=(19,17,14,13,13,8,1)\in\OO_3$. We use $3$-modular Ferrers graphs \cite{Andr1976} to illustrate the process of deriving $\mu$. See Figure~\ref{S-K map} below. For the readers' convenience, we have colored the inserted cells red for step 2.

We should remark that the original description of the Stockhofe-Keith map \cite{Kei-the, Sto} consists of only Steps 1 and 2 above. Thus the  map \cite{Kei-the, Sto} accounts for the following theorem.

\begin{theorem}[Theorem 3.1, \cite{KX}]  \label{KXtheorem}
$m$-regular partitions of any given $m$-length type are in bijection with $m$-flat partitions of the same $m$-length type.
\end{theorem}

Theorem~\ref{KXtheorem} with Remark~\ref{rem0} proves Theorem~\ref{KXthm}, so we will adopt the modified definition of the Stockhofe-Keith map  in this paper referring it to the map consisting of all the three steps above.

\begin{figure}[htp]
\begin{ferrers}

\addtext{-1.5}{-1.5}{Step 1: $\lambda=$}
\addcellrows{7+6+5+5+5+3+1}
\highlightcellbyletter{1}{1}{3}
\highlightcellbyletter{1}{2}{3}
\highlightcellbyletter{1}{3}{3}
\highlightcellbyletter{1}{4}{3}
\highlightcellbyletter{1}{5}{3}
\highlightcellbyletter{1}{6}{3}
\highlightcellbyletter{1}{7}{1}
\highlightcellbyletter{2}{1}{3}
\highlightcellbyletter{2}{2}{3}
\highlightcellbyletter{2}{3}{3}
\highlightcellbyletter{2}{4}{3}
\highlightcellbyletter{2}{5}{3}
\highlightcellbyletter{2}{6}{2}
\highlightcellbyletter{3}{1}{3}
\highlightcellbyletter{3}{2}{3}
\highlightcellbyletter{3}{3}{3}
\highlightcellbyletter{3}{4}{3}
\highlightcellbyletter{3}{5}{2}
\highlightcellbyletter{4}{1}{3}
\highlightcellbyletter{4}{2}{3}
\highlightcellbyletter{4}{3}{3}
\highlightcellbyletter{4}{4}{3}
\highlightcellbyletter{4}{5}{1}
\highlightcellbyletter{5}{1}{3}
\highlightcellbyletter{5}{2}{3}
\highlightcellbyletter{5}{3}{3}
\highlightcellbyletter{5}{4}{3}
\highlightcellbyletter{5}{5}{1}
\highlightcellbyletter{6}{1}{3}
\highlightcellbyletter{6}{2}{3}
\highlightcellbyletter{6}{3}{2}
\highlightcellbyletter{7}{1}{1}
\addtext{4}{-1.5}{$=$}
\putright
\addcellrow[0]{4}
\addcellrow[0]{4}
\addcellrow[1]{3}
\addcellrow[1]{3}
\addcellrow[1]{3}
\addcellrow[2]{2}
\highlightcellbyletter{1}{1}{3}
\highlightcellbyletter{1}{2}{3}
\highlightcellbyletter{1}{3}{3}
\highlightcellbyletter{1}{4}{3}
\highlightcellbyletter{2}{1}{3}
\highlightcellbyletter{2}{2}{3}
\highlightcellbyletter{2}{3}{3}
\highlightcellbyletter{2}{4}{3}
\highlightcellbyletter{3}{2}{3}
\highlightcellbyletter{3}{3}{3}
\highlightcellbyletter{3}{4}{3}
\highlightcellbyletter{4}{2}{3}
\highlightcellbyletter{4}{3}{3}
\highlightcellbyletter{4}{4}{3}
\highlightcellbyletter{5}{2}{3}
\highlightcellbyletter{5}{3}{3}
\highlightcellbyletter{5}{4}{3}
\highlightcellbyletter{6}{3}{3}
\highlightcellbyletter{6}{4}{3}
\addtext{3}{-1.5}{$+$}
\putright
\addcellrows{3+2+2+2+2+1+1}
\highlightcellbyletter{1}{1}{3}
\highlightcellbyletter{1}{2}{3}
\highlightcellbyletter{1}{3}{1}
\highlightcellbyletter{2}{1}{3}
\highlightcellbyletter{2}{2}{2}
\highlightcellbyletter{3}{1}{3}
\highlightcellbyletter{3}{2}{2}
\highlightcellbyletter{4}{1}{3}
\highlightcellbyletter{4}{2}{1}
\highlightcellbyletter{5}{1}{3}
\highlightcellbyletter{5}{2}{1}
\highlightcellbyletter{6}{1}{2}
\highlightcellbyletter{7}{1}{1}

\putbelow
\addtext{-1.5}{-1.5}{Step 2:}
\addcellrows{1+1+1+1+1+1}
\highlightcellbyletter{1}{1}{3}
\highlightcellbyletter{2}{1}{3}
\highlightcellbyletter{3}{1}{3}
\highlightcellbyletter{4}{1}{3}
\highlightcellbyletter{5}{1}{3}
\highlightcellbyletter{6}{1}{3}
\addtext{1.2}{-1.5}{$\xrightarrow{insert}$}
\putright
\addcellrows{3+2+2+2+2+1+1}
\highlightcellbyletter{1}{1}{3}
\highlightcellbyletter{1}{2}{3}
\highlightcellbyletter{1}{3}{1}
\highlightcellbyletter{2}{1}{3}
\highlightcellbyletter{2}{2}{2}
\highlightcellbyletter{3}{1}{3}
\highlightcellbyletter{3}{2}{2}
\highlightcellbyletter{4}{1}{3}
\highlightcellbyletter{4}{2}{1}
\highlightcellbyletter{5}{1}{3}
\highlightcellbyletter{5}{2}{1}
\highlightcellbyletter{6}{1}{2}
\highlightcellbyletter{7}{1}{1}
\addtext{2.2}{-1.5}{$=$}
\putright
\addcellrows{4+3+3+3+2+2+1+1}
\highlightcellbyletter{1}{1}{\redd{3}}
\highlightcellbyletter{1}{2}{3}
\highlightcellbyletter{1}{3}{3}
\highlightcellbyletter{1}{4}{1}
\highlightcellbyletter{2}{1}{\redd{3}}
\highlightcellbyletter{2}{2}{3}
\highlightcellbyletter{2}{3}{2}
\highlightcellbyletter{3}{1}{\redd{3}}
\highlightcellbyletter{3}{2}{3}
\highlightcellbyletter{3}{3}{2}
\highlightcellbyletter{4}{1}{\redd{3}}
\highlightcellbyletter{4}{2}{3}
\highlightcellbyletter{4}{3}{1}
\highlightcellbyletter{5}{1}{\redd{3}}
\highlightcellbyletter{5}{2}{\redd{3}}
\highlightcellbyletter{6}{1}{3}
\highlightcellbyletter{6}{2}{1}
\highlightcellbyletter{7}{1}{2}
\highlightcellbyletter{8}{1}{1}
\putbelow
\addcellrows{1+1+1+1+1+1}
\highlightcellbyletter{1}{1}{3}
\highlightcellbyletter{2}{1}{3}
\highlightcellbyletter{3}{1}{3}
\highlightcellbyletter{4}{1}{3}
\highlightcellbyletter{5}{1}{3}
\highlightcellbyletter{6}{1}{3}
\addtext{1.2}{-1.5}{$\xrightarrow{insert}$}
\putright
\addcellrows{4+3+3+3+2+2+1+1}
\highlightcellbyletter{1}{1}{3}
\highlightcellbyletter{1}{2}{3}
\highlightcellbyletter{1}{3}{3}
\highlightcellbyletter{1}{4}{1}
\highlightcellbyletter{2}{1}{3}
\highlightcellbyletter{2}{2}{3}
\highlightcellbyletter{2}{3}{2}
\highlightcellbyletter{3}{1}{3}
\highlightcellbyletter{3}{2}{3}
\highlightcellbyletter{3}{3}{2}
\highlightcellbyletter{4}{1}{3}
\highlightcellbyletter{4}{2}{3}
\highlightcellbyletter{4}{3}{1}
\highlightcellbyletter{5}{1}{3}
\highlightcellbyletter{5}{2}{3}
\highlightcellbyletter{6}{1}{3}
\highlightcellbyletter{6}{2}{1}
\highlightcellbyletter{7}{1}{2}
\highlightcellbyletter{8}{1}{1}
\addtext{2.8}{-1.5}{$=$}
\putright
\addcellrows{5+4+4+3+3+2+2+1+1}
\addtext{3}{-1.5}{$\cdots$}
\highlightcellbyletter{1}{1}{\redd{3}}
\highlightcellbyletter{1}{2}{3}
\highlightcellbyletter{1}{3}{3}
\highlightcellbyletter{1}{4}{3}
\highlightcellbyletter{1}{5}{1}
\highlightcellbyletter{2}{1}{\redd{3}}
\highlightcellbyletter{2}{2}{3}
\highlightcellbyletter{2}{3}{3}
\highlightcellbyletter{2}{4}{2}
\highlightcellbyletter{3}{1}{\redd{3}}
\highlightcellbyletter{3}{2}{3}
\highlightcellbyletter{3}{3}{3}
\highlightcellbyletter{3}{4}{2}
\highlightcellbyletter{4}{1}{\redd{3}}
\highlightcellbyletter{4}{2}{\redd{3}}
\highlightcellbyletter{4}{3}{\redd{3}}
\highlightcellbyletter{5}{1}{3}
\highlightcellbyletter{5}{2}{3}
\highlightcellbyletter{5}{3}{1}
\highlightcellbyletter{6}{1}{3}
\highlightcellbyletter{6}{2}{3}
\highlightcellbyletter{7}{1}{3}
\highlightcellbyletter{7}{2}{1}
\highlightcellbyletter{8}{1}{2}
\highlightcellbyletter{9}{1}{1}

\putbelow
\addtext{-1.5}{-1.5}{Step 3:}
\addcellrows{5+5+4+4+3+3+2+2+2+1+1}
\highlightcellbyletter{1}{1}{3}
\highlightcellbyletter{1}{2}{3}
\highlightcellbyletter{1}{3}{3}
\highlightcellbyletter{1}{4}{3}
\highlightcellbyletter{1}{5}{3}
\highlightcellbyletter{2}{1}{3}
\highlightcellbyletter{2}{2}{3}
\highlightcellbyletter{2}{3}{3}
\highlightcellbyletter{2}{4}{3}
\highlightcellbyletter{2}{5}{1}
\highlightcellbyletter{3}{1}{3}
\highlightcellbyletter{3}{2}{3}
\highlightcellbyletter{3}{3}{3}
\highlightcellbyletter{3}{4}{2}
\highlightcellbyletter{4}{1}{3}
\highlightcellbyletter{4}{2}{3}
\highlightcellbyletter{4}{3}{3}
\highlightcellbyletter{4}{4}{2}
\highlightcellbyletter{5}{1}{3}
\highlightcellbyletter{5}{2}{3}
\highlightcellbyletter{5}{3}{3}
\highlightcellbyletter{6}{1}{3}
\highlightcellbyletter{6}{2}{3}
\highlightcellbyletter{6}{3}{1}
\highlightcellbyletter{7}{1}{3}
\highlightcellbyletter{7}{2}{3}
\highlightcellbyletter{8}{1}{3}
\highlightcellbyletter{8}{2}{3}
\highlightcellbyletter{9}{1}{3}
\highlightcellbyletter{9}{2}{1}
\highlightcellbyletter{10}{1}{2}
\highlightcellbyletter{11}{1}{1}
\addtext{3.5}{-2}{$\xrightarrow{conjugate}$}
\addtext{8.3}{-2.2}{$\mu=(11,10,9,9,8,8,6,5,5,4,4,2,2,1,1)$}
\end{ferrers}
\caption{Three steps to get $\phi(\lambda)=\mu$}
\label{S-K map}
\end{figure}

\subsection{Proof of Theorem~\ref{thm: one para refine}}

In view of the proof of Theorem~\ref{KXthm} in \cite{KX}, all it remains is to examine the map $\phi$ with respect to the extra parameter $z_m$. Suppose $\lambda\in\OO_m$, and the largest part of $\beta(\lambda)$ is $b_1$, then we have $\left\lfloor\dfrac{b_1}{m}\right\rfloor=\asc(\lambda)$, according to the definition of $m$-flat partitions. Next, during step 2, we insert columns of $m\sigma$ into $\tau$, and each insertion will give rise to a new part in $\tau$ that is divisible by $m$, therefore we see $\dfrac{\lambda_1-b_1}{m}=s_m(\mu)$. The above discussion gives us
\begin{equation*}
s_m(\mu)=\dfrac{\lambda_1-b_1}{m}=\left\lfloor\dfrac{\lambda_1}{m}\right\rfloor-\left\lfloor\dfrac{b_1}{m}\right\rfloor=\left\lfloor\dfrac{\lambda_1}{m}\right\rfloor-\asc(\lambda),
\end{equation*}
as desired. \hfill \qed

\begin{remark}
For $\lambda\in\OO_m$, let $\phi(\lambda)=\mu$. Then, the extra parameter tracked by $z_m$ gives us
\begin{align*}
\mu_1=s(\mu)+s_m(\mu)=s(\mu)+\left\lfloor\frac{\lambda_1}{m}\right\rfloor-\asc(\lambda)=\ell(\lambda)+\left\lfloor\frac{\lambda_1}{m}\right\rfloor-\asc(\lambda),
\end{align*}
which has previously been derived by Keith \cite[Theorem~6]{Kei-the} as well (he used $f_{\lambda}$ instead of our $\asc(\lambda)$). Moreover, this refinement is reminiscent of Sylvester's bijection for proving Euler's theorem, in which case $m=2$ and we always have $\asc(\lambda)=0$, see for example Theorem 1 (item 4) in \cite{Zen}.
\end{remark}

\section{A lecture hall theorem for \texorpdfstring{$m$-falling}{} partitions}\label{sec3: m-falling}


We will first handle the case with a single residue class. Let us fix $c, 1\le c \le m-1$. For $n\ge 1$, let
\begin{align*}
\OO_{c,m}&:=\{\lambda\in\OO_m:\lambda_i\equiv c\pmod{m}, \text{ for all $i$}\},\\
\OO_{c,m}^n&:=\{\lambda\in\OO_{c,m}:\lambda_1<nm\},\\
\D_{c,m}&:=\left\{\lambda\in\D_m:\sbf(\lambda)=(\overbrace{0,\ldots,0}^{c-1},a,0\ldots,0) \text{ for some }a> 0, \text{ or }|\lambda|=0 \right\},\\
\LL_{c,m}^n&:=\left\{\lambda\in\D_{c,m}: l(\lambda)\le \left\lfloor\dfrac{n+1}{2}\right\rfloor (m-2) +n \text{ and} \right.\\
&\phantom{:=}\left.
\qquad \qquad \qquad \frac{\lambda_{km+c}}{n-2k}\ge\frac{\lambda_{km+m}}{n-2k-1} \ge\frac{\lambda_{(k+1)m+c}}{n-2k-2} \text{ for } 0\le k<\left\lceil\frac{n}{2}\right\rceil\right\},
\end{align*}
where $l(\lambda)$ is the number of nonzero parts in $\lambda$, and we make the convention that for a fraction $\dfrac{p}{q}$, $q=0$ forces $p=0$. 


\begin{theorem}\label{thm: mc-lecture hall}
For any $n\ge 1$,
\begin{align}\label{id: mc-lechall}
\sum_{\lambda\in\OO_{c,m}^{n}}z^{\ell(\lambda)}q^{|\lambda|}=\sum_{\mu\in\LL_{c,m}^{n}}z^{s(\mu)}q^{|\mu|}=\frac{1}{(1-zq^c)(1-zq^{m+c})\cdots(1-zq^{(n-1)m+c})}.
\end{align}
\end{theorem}

The above result can be viewed as a finite (or ``lecture hall'') version of the following Pak-Postnikov's $(m,c)$-generalization \cite[Theorem~1]{PP} of Euler's partition theorem since $\lim_{n\to \infty} \OO_{c,m}^n\to \OO_{c,m}$ and $\lim_{n\to \infty} \LL_{c,m}^n \to \D_{c,m}$.

\begin{theorem}
For $n\ge 1, m\ge 2$ and $1\le c\le m-1$, the number of partitions of $n$ into parts congruent to $c$ modulo $m$, equals the number of partitions of $n$ with exactly $c$ parts of maximal size, $m-c$ (if any) second by size parts, $c$ (if any) third by size parts, etc.
\end{theorem}

\begin{proof}[Proof of Theorem~\ref{thm: mc-lecture hall}]

Since a partition $\lambda\in\OO_{c,m}^n$ has all its parts congruent to $c$ modulo $m$, with the largest part $\lambda_1<nm$, we see that $\ell(\lambda)=\ell_c(\lambda)=l(\lambda)$ and we have
\begin{align*}
\sum_{\lambda\in\OO_{c,m}^{n}}z^{\ell(\lambda)}q^{|\lambda|}=\frac{1}{(1-zq^c)(1-zq^{m+c})\cdots(1-zq^{(n-1)m+c})}.
\end{align*}

It remains to prove the first equality. We achieve this by constructing a bijection $\varphi_n$ from $\OO_{c,m}^n$ to $\LL_{c,m}^n$ that is weight-preserving and sends $\ell(\lambda)$ to $s(\varphi_n(\lambda))$. This bijection is based on the third author's original idea from \cite{Yee}, which was later generalized to deal with the $\ell$-sequence version in \cite{SY}. We modify it to suit the current settings.

Define
$\varphi_n: \OO_{c,m}^n \longrightarrow \LL_{c,m}^n$ as follows.
For $\lambda\in \OO_{c,m}^n$, 
let $\mu$ be the sequence obtained from the empty sequence $(0,0,\ldots)$ by recursively inserting the parts of $\lambda$ in nonincreasing order according to the following insertion procedure. 
We define $\varphi_n(\lambda)=\mu$.

\begin{framed}
\begin{center}
{\bf Insertion procedure}
\end{center}

\noindent Let $(\mu_1, \mu_2,\ldots) \in \LL_{c,m}^n$.  To insert $km+c$ into $(\mu_1,\mu_2,\ldots)$, set $j=0$. \hfill



If $j<k$ and $\dfrac{\mu_{mj+c}+1}{n-2j}\ge \dfrac{\mu_{mj+m}+1}{n-2j-1}$, \hfill (Test I)

\hskip 0.2in then increase $j$ by $1$ and go to (Test I); \hfill

otherwise stop testing and return
$$(\mu_1,\mu_2,\ldots)+(\overbrace{1,\ldots,1}^{jm},\overbrace{k-j+1,\ldots,k-j+1}^{c},\overbrace{k-j,\ldots,k-j}^{m-c},0,0,\ldots).$$

\end{framed}

The effect of this insertion is that we use up a complete part $km+c$, so the weight of the sequence $(\mu_1,\mu_2,\ldots)$ and its $m$-alternating sum are increased by $km+c$ and $1$, respectively. Also, it can be checked easily that the returned sequence satisfies the condition for $\LL_{c,m}^n$. We omit the details.

The map $\varphi_n$ is indeed invertible since the parts of $\lambda$ were inserted in nonincreasing order, i.e., from the largest to smallest.  If the parts are not inserted in this order, $\varphi_n$ is not necessarily invertible.  The inverse of $\varphi_n$, namely $\psi_n$, can be described similarly in this algorithmic fashion. For a given partition $\mu\in\LL_{c,m}^n$ with $s(\mu)=k$, define $\psi_n(\mu)$ to be the sequence $\lambda=(\lambda_1,\ldots,\lambda_k,0,0,\ldots)$ obtained from the empty sequence $(0,0,\ldots)$ by adding nondecreasing parts one at a time that are derived from peeling off partially or entirely certain parts of $\mu$ according to the following deletion procedure.

\begin{framed}
\begin{center}
{\bf Deletion procedure}
\end{center}
\noindent Let $(\mu_1,\mu_2,\ldots)\neq (0,0,\ldots)$ be in $\LL_{c,m}^n$. Set $k=0$ and $j=0$. \hfill




If $(\mu_1,\mu_2,\ldots)-(\overbrace{1,\ldots,1}^{jm},\overbrace{k-j+1,\ldots,k-j+1}^{c},\overbrace{k-j,\ldots,k-j}^{m-c},0,0,\ldots) \in\LL_{c,m}^n$,
 (Test D)

\hskip 0.2in  then stop testing and return $km+c$ with \hfill

\qquad \qquad   $\mu=(\mu_1,\mu_2,\ldots)-(\overbrace{1,\ldots,1}^{jm},\overbrace{k-j+1,\ldots, k-j+1}^{c},\overbrace{k-j,\ldots, k-j}^{m-c},0,0,\ldots)$; \hfill

otherwise,  if  $j<k$, then increase $j$ by $1$ and  go to (Test D); \hfill

 \qquad \qquad  otherwise increase $k$ by $1$, set $j=0$ and go to (Test D).\hfill
\end{framed}

The effect of this deletion is that the weight of the sequence $(\mu_1,\mu_2,\ldots)$ and its $m$-alternating sum are decreased by $km+c$ and $1$, respectively. Also, it should be noted that this deletion process must stop after a finite number of steps. Since $(\mu_1, \mu_2,\ldots )\neq (0,0,\ldots)$ belongs to $\LL_{c,m}^n$, there must be  $i$ such that
\begin{equation*}
\mu_{i m+c}>\mu_{i m+m}. 
\end{equation*}
Let $j$ be the largest such $i$. Then, $\mu_{l}=0$ for any $l>jm+m$ and
\begin{equation*}
(\mu_1, \mu_2, \ldots)- (\overbrace{1,\ldots,1}^{jm},\overbrace{k-j+1,\ldots, k-j+1}^{c},\overbrace{k-j,\ldots, k-j}^{m-c},0,0,\ldots) \in \LL_{c,m}^n,
\end{equation*}
which shows such $j$ must pass (Test D).

To finish the proof, we make the following claims about $\varphi_n$ and $\psi_n$ without giving the proofs, since all of them are essentially the same as those found in \cite{Yee}, which is the case when $m=2$ and $c=1$.
\begin{itemize}
	\item Each insertion outputs a new $\mu\in\LL_{c,m}^n$, and in particular, $\varphi_n$ is well-defined.
	\item Each deletion outputs a new $\lambda\in\OO_{c,m}^n$, and in particular, $\psi_n$ is well-defined.
	\item The deletion procedure reverses the insertion procedure, consequently $\psi_n$ is the inverse of $\varphi_n$. 
\end{itemize}
\end{proof}

Before we move on, we provide an example for the insertion procedure.
\begin{example}
Let $m=3$, $c=2$,  $n=5$, and $\mu=(3,3,2,0,0,\ldots) \in \LL_{2,3}^5$. We insert $8$ into $\mu$ as follows. Note that
\begin{align*}
\frac{\mu_2+1}{5}=\frac{4}{5}\ge \frac{\mu_{3}+1}{4}=\frac{3}{4} \quad \text{but } \quad \frac{\mu_5+1}{3}=\frac{1}{3} \not\ge \frac{\mu_{6}+1}{2}=\frac{1}{2}.
\end{align*}
So we get
\begin{equation*}
(3,3,2,0,0,\ldots)+ (1,1,1,2,2,1)=(4,4,3,2,2,1).
\end{equation*}

\end{example}

In Figure~\ref{insertion}, we illustrate the process of getting $\mu$ by applying $\varphi_5$ to $\lambda=(11, 11, 8, 8, 8, 5, 5) \in \OO_{2,3}^{5}$ using $3$-modular Ferrers' graphs. Newly inserted cells after each step are colored red.

\begin{figure}[htp]
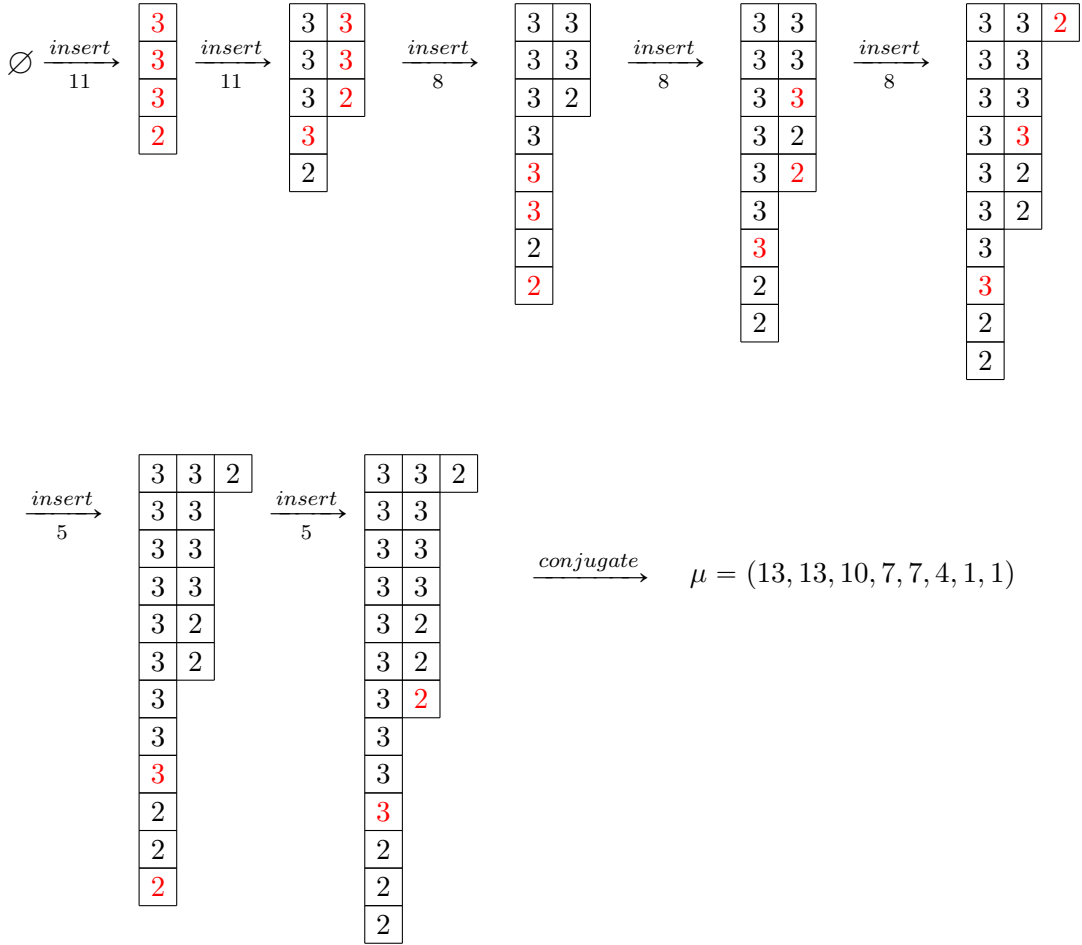

\begin{ferrers}
\addtext{-1.0}{-0.8}{$\emptyset\xrightarrow[11]{insert}$}
\addcellrows{1+1+1+1}
\highlightcellbyletter{1}{1}{\redd{3}}
\highlightcellbyletter{2}{1}{$\redd{3}$}
\highlightcellbyletter{3}{1}{$\redd{3}$}
\highlightcellbyletter{4}{1}{$\redd{2}$}

\addtext{1.25}{-.8}{$\xrightarrow[11]{insert}$}
\putright
\addcellrows{2+2+2+1+1}
\highlightcellbyletter{1}{1}{3}
\highlightcellbyletter{2}{1}{3}
\highlightcellbyletter{3}{1}{$3$}
\highlightcellbyletter{5}{1}{$2$}
\highlightcellbyletter{4}{1}{$\redd{3}$}
\highlightcellbyletter{1}{2}{$\redd{3}$}
\highlightcellbyletter{2}{2}{$\redd{3}$}
\highlightcellbyletter{3}{2}{$\redd{2}$}

\addtext{2}{-.8}{$\xrightarrow[8]{insert}$}
\putright
\addcellrows{2+2+2+1+1+1+1+1}
\highlightcellbyletter{1}{1}{3}
\highlightcellbyletter{2}{1}{3}
\highlightcellbyletter{3}{1}{$3$}
\highlightcellbyletter{4}{1}{$3$}
\highlightcellbyletter{7}{1}{$2$}
\highlightcellbyletter{1}{2}{$3$}
\highlightcellbyletter{2}{2}{$3$}
\highlightcellbyletter{3}{2}{$2$}
\highlightcellbyletter{5}{1}{$\redd{3}$}
\highlightcellbyletter{6}{1}{$\redd{3}$}
\highlightcellbyletter{8}{1}{$\redd{2}$}

\addtext{2}{-.8}{$\xrightarrow[8]{insert}$}
\putright
\addcellrows{2+2+2+2+2+1+1+1+1}
\highlightcellbyletter{1}{1}{3}
\highlightcellbyletter{2}{1}{3}
\highlightcellbyletter{3}{1}{$3$}
\highlightcellbyletter{4}{1}{$3$}
\highlightcellbyletter{5}{1}{$3$}
\highlightcellbyletter{1}{2}{$3$}
\highlightcellbyletter{2}{2}{$3$}
\highlightcellbyletter{4}{2}{$2$}
\highlightcellbyletter{6}{1}{${3}$}
\highlightcellbyletter{8}{1}{${2}$}
\highlightcellbyletter{9}{1}{${2}$}
\highlightcellbyletter{7}{1}{$\redd{3}$}
\highlightcellbyletter{3}{2}{$\redd{3}$}
\highlightcellbyletter{5}{2}{$\redd{2}$}

\addtext{2}{-.8}{$\xrightarrow[8]{insert}$}

\putright

\addcellrows{3+2+2+2+2+2+1+1+1+1}
\highlightcellbyletter{1}{1}{3}
\highlightcellbyletter{2}{1}{3}
\highlightcellbyletter{3}{1}{$3$}
\highlightcellbyletter{4}{1}{$3$}
\highlightcellbyletter{5}{1}{$3$}
\highlightcellbyletter{1}{2}{$3$}
\highlightcellbyletter{2}{2}{$3$}
\highlightcellbyletter{3}{2}{$3$}
\highlightcellbyletter{6}{1}{${3}$}
\highlightcellbyletter{7}{1}{${3}$}
\highlightcellbyletter{10}{1}{${2}$}
\highlightcellbyletter{9}{1}{${2}$}
\highlightcellbyletter{4}{2}{$\redd{3}$}
\highlightcellbyletter{5}{2}{${2}$}
\highlightcellbyletter{8}{1}{$\redd{3}$}
\highlightcellbyletter{6}{2}{$2$}
\highlightcellbyletter{1}{3}{$\redd{2}$}

\putbelow
\addtext{-1.0}{-.8}{$\xrightarrow[5]{insert}$}

\addcellrows{3+2+2+2+2+2+1+1+1+1+1+1}
\highlightcellbyletter{1}{1}{3}
\highlightcellbyletter{2}{1}{3}
\highlightcellbyletter{3}{1}{$3$}
\highlightcellbyletter{4}{1}{$3$}
\highlightcellbyletter{5}{1}{$3$}
\highlightcellbyletter{1}{2}{$3$}
\highlightcellbyletter{2}{2}{$3$}
\highlightcellbyletter{3}{2}{$3$}
\highlightcellbyletter{6}{1}{${3}$}
\highlightcellbyletter{7}{1}{${3}$}
\highlightcellbyletter{8}{1}{${3}$}
\highlightcellbyletter{10}{1}{${2}$}
\highlightcellbyletter{4}{2}{${3}$}
\highlightcellbyletter{5}{2}{${2}$}
\highlightcellbyletter{11}{1}{${2}$}
\highlightcellbyletter{6}{2}{${2}$}
\highlightcellbyletter{1}{3}{${2}$}
\highlightcellbyletter{9}{1}{$\redd{3}$}
\highlightcellbyletter{12}{1}{$\redd{2}$}

\addtext{2.25}{-.8}{$\xrightarrow[5]{insert}$}
\putright
\addcellrows{3+2+2+2+2+2+2+1+1+1+1+1+1}
\highlightcellbyletter{1}{1}{3}
\highlightcellbyletter{2}{1}{3}
\highlightcellbyletter{3}{1}{$3$}
\highlightcellbyletter{4}{1}{$3$}
\highlightcellbyletter{5}{1}{$3$}
\highlightcellbyletter{1}{2}{$3$}
\highlightcellbyletter{2}{2}{$3$}
\highlightcellbyletter{3}{2}{$3$}
\highlightcellbyletter{6}{1}{${3}$}
\highlightcellbyletter{7}{1}{${3}$}
\highlightcellbyletter{8}{1}{${3}$}
\highlightcellbyletter{9}{1}{${3}$}
\highlightcellbyletter{4}{2}{${3}$}
\highlightcellbyletter{5}{2}{${2}$}
\highlightcellbyletter{13}{1}{${2}$}
\highlightcellbyletter{6}{2}{${2}$}
\highlightcellbyletter{1}{3}{${2}$}
\highlightcellbyletter{11}{1}{${2}$}
\highlightcellbyletter{12}{1}{${2}$}
\highlightcellbyletter{10}{1}{$\redd{3}$}
\highlightcellbyletter{7}{2}{$\redd{2}$}

\putright
\addtext{0}{-1.5}{$\xrightarrow{conjugate}$}
\addtext{3.5}{-1.6}{$\mu=(13,13,10,7,7, 4, 1, 1)$}
\end{ferrers}
\caption{$\varphi_5((11,11,8,8,8, 5,5))=(13,13,10,7, 7, 4,1,1)$}
\label{insertion}
\end{figure}

As we will see, the bijection $\varphi_n$ plays a crucial role in our proof of Theorem~\ref{thm:mfallingLH}. We need a few more definitions.

\begin{Def}
For a partition $\lambda\in\D_m$, we define two local statistics ``\textbf{f}irst \textbf{b}igger'' and ``\textbf{l}ast \textbf{b}igger''. For each $i,\:0\le i\le\left\lfloor\dfrac{l(\lambda)-1}{m}\right\rfloor$, where $l(\lambda)$ is the number of nonzero parts in $\lambda$, suppose
$$
\lambda_{im+1}=\cdots=\lambda_{im+j}>\lambda_{im+j+1}\ge\cdots\ge\lambda_{im+k}>\lambda_{im+k+1}=\cdots=\lambda_{im+m}.
$$
Then we let $\fb_i=j,\:\lb_i=k$.
\end{Def}

Note that since $\lambda\in\D_m$, such $j$ and $k$ must exist and $j\le k$, so $\fb_i$ and $\lb_i$ are well-defined.

\begin{Def}\label{m-falling LH}
Fix a positive integer $n$. For a partition $\lambda\in\D_m$, we call it an \emph{$m$-falling lecture hall partition of order $n$}, if $l(\lambda)\le \left\lfloor\dfrac{n+1}{2}\right\rfloor(m-2)+n$, and the following two conditions hold.
\begin{enumerate}
	\item For $i,\:0\le i<\left\lfloor\dfrac{l(\lambda)-1}{m}\right\rfloor$, $\lb_i\le \fb_{i+1}$.
	\item \begin{align*}
	\dfrac{\lambda_1}{n}\ge\dfrac{\lambda_m}{n-1}\ge\dfrac{\lambda_{m+1}}{n-2}\ge\dfrac{\lambda_{2m}}{n-3}\ge\cdots\ge\dfrac{\lambda_{(k-1)m+1}}{2}\ge\dfrac{\lambda_{km}}{1}, \text{ if $n=2k$,}\\
	\dfrac{\lambda_1}{n}\ge\dfrac{\lambda_m}{n-1}\ge\dfrac{\lambda_{m+1}}{n-2}\ge\dfrac{\lambda_{2m}}{n-3}\ge\cdots\ge\dfrac{\lambda_{km+1}}{1}\ge\dfrac{\lambda_{km+m}}{0}, \text{ if $n=2k+1$.}
	\end{align*}
\end{enumerate}
We denote the set of all $m$-falling lecture hall partitions of order $n$ as $\LL_{m\searrow}^n$.
\end{Def}
\begin{remark}
Partitions in $\D_m$ satisfying condition (1) above are said to be of $m$-alternating type in \cite{Kei-the}.
\end{remark}

Recall  the definition of $m$-falling regular partitions. A partition is $m$-falling regular if the parts are not multiples of $m$ and their positive residues are nonincreasing.

For a chosen vector $\vbf=(v_1,v_2,\ldots,v_{m-1})$, let
\begin{align*}
\OO_{m\searrow}^{\vbf}&:=\{\lambda\in\OO_{m\searrow}^n:\lbf(\lambda)=\vbf\},\\
\LL_{m\searrow}^{\vbf}&:=\{\mu\in\LL_{m\searrow}^{n}:\sbf(\mu)=\vbf\}.
\end{align*}


\subsection{Proof of Theorem~\ref{thm:mfallingLH}}

For some $c, 1\le c \le m-1$ and a vector $\vbf=(v_1,v_2,\ldots,v_{m-1})$, we consider two embeddings,
$$f_{\vbf}:\OO_{m\searrow}^{\vbf}\hookrightarrow\OO_{c,m}^n \text{ and } g_{\vbf}:\LL_{m\searrow}^{\vbf}\hookrightarrow\LL_{c,m}^n,$$
such that $\ell(\lambda)=\ell(f_{\vbf}(\lambda))$ and $s(\mu)=s(g_{\vbf}(\mu))$. To be precise, for a given partition $\lambda$, both $f_{\vbf}$ and $g_{\vbf}$ change the residue of each part of $\lambda$ mod $m$ to be uniformly the predetermined value $c$. In terms of the corresponding $m$-modular Ferrers' graph, the two maps keep all the cells labelled $m$, but relabel all the remaining cells as $c$. So in general, neither of these two maps preserves the weight of the partition, but they do keep the number of cells in their $m$-modular Ferrers' graphs the same.

Moreover, the given vector $\vbf$ and the $m$-falling condition uniquely determine the preimage of any partition in $f_{\vbf}(\OO_{m\searrow}^{\vbf})$. Similarly, the condition (1) in Definition~\ref{m-falling LH} together with $\vbf$ dictate the preimage of any partition in $g_{\vbf}(\LL_{m\searrow}^{\vbf})$. This entitles us to define a bijection $$\phi_n=g_{\vbf}^{-1}\circ\varphi_n\circ f_{\vbf}: \OO_{m\searrow}^n\longrightarrow \LL_{m\searrow}^n,$$
where $\vbf$ is the $m$-length type of the partition it acts on.

It has been proved in Theorem~\ref{thm: mc-lecture hall} that $\varphi_n$ is a bijection satisfying $\ell(\lambda)=s(\varphi_n(\lambda))$, and the discussion above shows that both $f_{\vbf}$ and $g_{\vbf}$ are invertible. Consequently, we see that $\phi_n$ is indeed a bijection such that $\lbf(\lambda)=\sbf(\phi_n(\lambda))$ for any $\lambda\in\OO_{m\searrow}^n$, and we complete the proof. \hfill \qed

\begin{example}
As an illustrative example, we take $m=3,n=3$ and fix the vector $\vbf=(3,2)$. In Table~\ref{tab: phi_3}, we list out all $21$ partitions $\lambda$ in $\OO_{3\searrow}^{(3,2)}$ with $\lambda_1<nm=9$, as well as all $21$ partitions $\mu$ in $\LL_{3\searrow}^{(3,2)}$ with $l(\mu)\le 5$. They are matched up via our map $\phi_3$. The derivation of one particular partition $(8,5,5,2,2)$ from $(5,5,4,4,4)$ using $3$-modular Ferrers' graphs is detailed in Figure~\ref{phi map}.
\end{example}

\begin{figure}[htp]
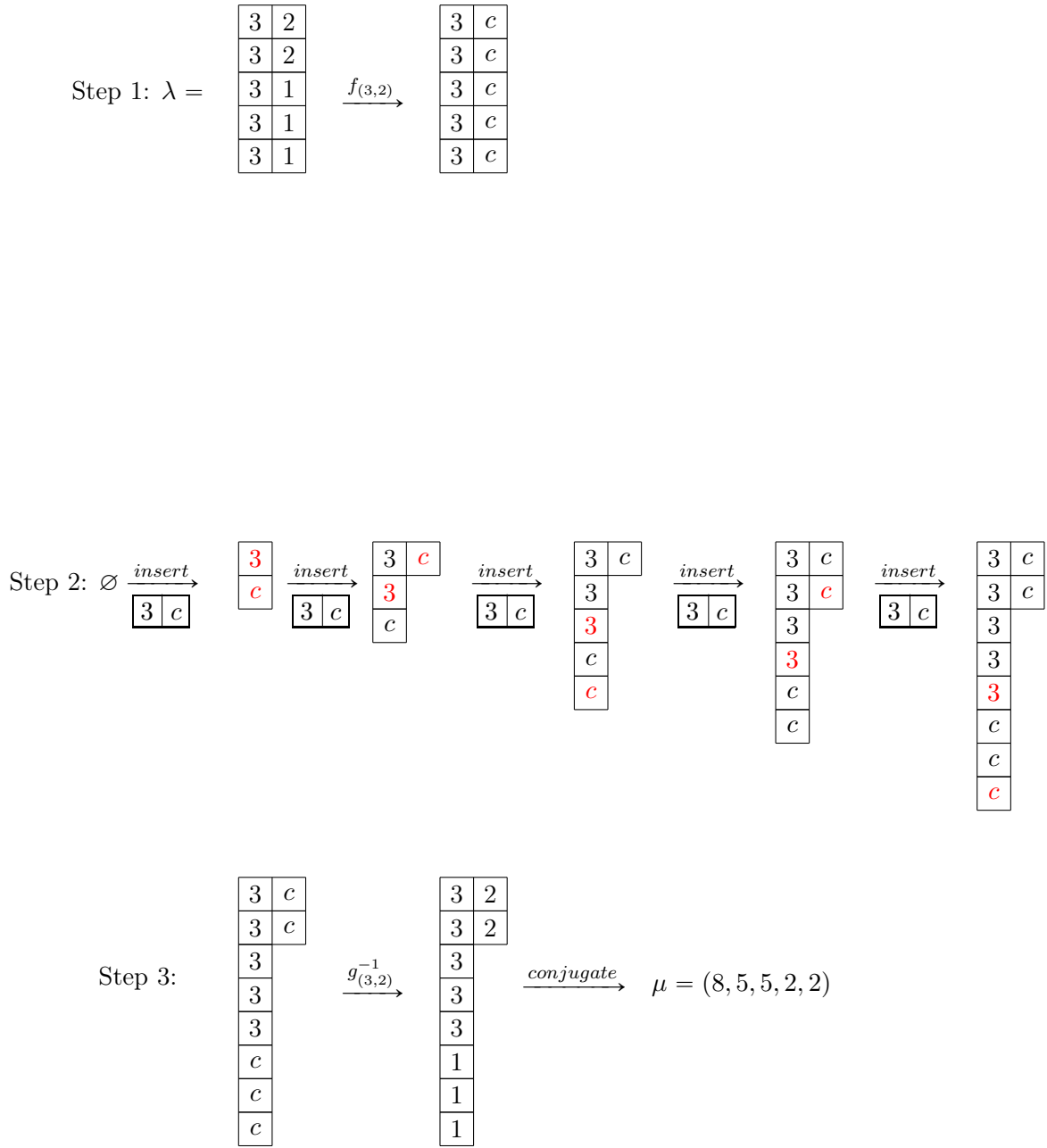

\begin{ferrers}
\addtext{-1.5}{-1.3}{Step 1: $\lambda=$}
\addcellrows{2+2+2+2+2}
\highlightcellbyletter{1}{1}{3}
\highlightcellbyletter{2}{1}{3}
\highlightcellbyletter{3}{1}{3}
\highlightcellbyletter{4}{1}{3}
\highlightcellbyletter{5}{1}{3}
\highlightcellbyletter{1}{2}{2}
\highlightcellbyletter{2}{2}{2}
\highlightcellbyletter{3}{2}{1}
\highlightcellbyletter{4}{2}{1}
\highlightcellbyletter{5}{2}{1}
\addtext{2}{-1.3}{$\xrightarrow{f_{(3,2)}}$}
\putright
\addcellrows{2+2+2+2+2}
\highlightcellbyletter{1}{1}{3}
\highlightcellbyletter{2}{1}{3}
\highlightcellbyletter{3}{1}{3}
\highlightcellbyletter{4}{1}{3}
\highlightcellbyletter{5}{1}{3}
\highlightcellbyletter{1}{2}{$c$}
\highlightcellbyletter{2}{2}{$c$}
\highlightcellbyletter{3}{2}{$c$}
\highlightcellbyletter{4}{2}{$c$}
\highlightcellbyletter{5}{2}{$c$}
\putbelow
\addtext{-2}{-0.8}{Step 2: $\varnothing\xrightarrow[\boxed{3\ \ c}]{insert}$}
\addcellrows{1+1}
\highlightcellbyletter{1}{1}{$\redd{3}$}
\highlightcellbyletter{2}{1}{$\redd{c}$}
\addtext{1.25}{-.8}{$\xrightarrow[\boxed{3\ \ c}]{insert}$}
\putright
\addcellrows{2+1+1}
\highlightcellbyletter{1}{1}{3}
\highlightcellbyletter{2}{1}{$\redd{3}$}
\highlightcellbyletter{3}{1}{$c$}
\highlightcellbyletter{1}{2}{$\redd{c}$}
\addtext{2}{-.8}{$\xrightarrow[\boxed{3\ \ c}]{insert}$}
\putright
\addcellrows{2+1+1+1+1}
\highlightcellbyletter{1}{1}{3}
\highlightcellbyletter{2}{1}{3}
\highlightcellbyletter{3}{1}{$\redd{3}$}
\highlightcellbyletter{4}{1}{$c$}
\highlightcellbyletter{5}{1}{$\redd{c}$}
\highlightcellbyletter{1}{2}{$c$}
\addtext{2}{-.8}{$\xrightarrow[\boxed{3\ \ c}]{insert}$}
\putright
\addcellrows{2+2+1+1+1+1}
\highlightcellbyletter{1}{1}{3}
\highlightcellbyletter{2}{1}{3}
\highlightcellbyletter{3}{1}{3}
\highlightcellbyletter{4}{1}{$\redd{3}$}
\highlightcellbyletter{5}{1}{$c$}
\highlightcellbyletter{6}{1}{$c$}
\highlightcellbyletter{1}{2}{$c$}
\highlightcellbyletter{2}{2}{$\redd{c}$}
\addtext{2}{-.8}{$\xrightarrow[\boxed{3\ \ c}]{insert}$}
\putright
\addcellrows{2+2+1+1+1+1+1+1}
\highlightcellbyletter{1}{1}{3}
\highlightcellbyletter{2}{1}{3}
\highlightcellbyletter{3}{1}{3}
\highlightcellbyletter{4}{1}{3}
\highlightcellbyletter{5}{1}{$\redd{3}$}
\highlightcellbyletter{6}{1}{$c$}
\highlightcellbyletter{7}{1}{$c$}
\highlightcellbyletter{8}{1}{$\redd{c}$}
\highlightcellbyletter{1}{2}{$c$}
\highlightcellbyletter{2}{2}{$c$}

\putbelow
\addtext{-1.5}{-1.5}{Step 3:}
\addcellrows{2+2+1+1+1+1+1+1}
\highlightcellbyletter{1}{1}{3}
\highlightcellbyletter{2}{1}{3}
\highlightcellbyletter{3}{1}{3}
\highlightcellbyletter{4}{1}{3}
\highlightcellbyletter{5}{1}{3}
\highlightcellbyletter{6}{1}{$c$}
\highlightcellbyletter{7}{1}{$c$}
\highlightcellbyletter{8}{1}{$c$}
\highlightcellbyletter{1}{2}{$c$}
\highlightcellbyletter{2}{2}{$c$}
\addtext{2}{-1.5}{$\xrightarrow{g_{(3,2)}^{-1}}$}
\putright
\addcellrows{2+2+1+1+1+1+1+1}
\highlightcellbyletter{1}{1}{3}
\highlightcellbyletter{2}{1}{3}
\highlightcellbyletter{3}{1}{3}
\highlightcellbyletter{4}{1}{3}
\highlightcellbyletter{5}{1}{3}
\highlightcellbyletter{6}{1}{1}
\highlightcellbyletter{7}{1}{1}
\highlightcellbyletter{8}{1}{1}
\highlightcellbyletter{1}{2}{2}
\highlightcellbyletter{2}{2}{2}
\addtext{2}{-1.5}{$\xrightarrow{conjugate}$}
\addtext{4.5}{-1.6}{$\mu=(8,5,5,2,2)$}
\setbase{0}{0}
\addline{-1.14}{-8.78}{-1.14}{-9.25}
\addline{1.25}{-8.78}{1.25}{-9.25}
\addline{4}{-8.78}{4}{-9.25}
\addline{7}{-8.78}{7}{-9.25}
\addline{10}{-8.78}{10}{-9.25}
\end{ferrers}
\caption{$\phi_3((5,5,4,4,4))=(8,5,5,2,2)$}
\label{phi map}
\end{figure}

\begin{table}[tbp]\caption{$\phi_3:\OO_{3\searrow}^{(3,2)}\longrightarrow\LL_{3\searrow}^{(3,2)}$}\label{tab: phi_3}
\centering
\begin{tabular}{c||l}
\hline
 $\OO_{3\searrow}^{(3,2)}$ & \quad $\LL_{3\searrow}^{(3,2)}$\\
\hline
$8,8,7,7,7$ & $15,12,10$ \\
$8,8,7,7,4$ & $14,11,9$ \\
$8,8,7,7,1$ & $13,10,8$ \\
$8,8,7,4,4$ & $12,9,8,1,1$ \\
$8,8,7,4,1$ & $12,9,7$ \\
$8,8,7,1,1$ & $11,8,6$ \\
$8,8,4,4,4$ & $11,8,7,1,1$ \\
$8,8,4,4,1$ & $10,7,6,1,1$ \\
$8,8,4,1,1$ & $10,7,5$ \\
$8,8,1,1,1$ & $9,6,4$ \\
$8,5,4,4,4$ & $9,6,6,2,2$ \\
$8,5,4,4,1$ & $9,6,5,1,1$ \\
$8,5,4,1,1$ & $8,5,4,1,1$ \\
$8,5,1,1,1$ & $8,5,3$ \\
$8,2,1,1,1$ & $7,4,2$ \\
$5,5,4,4,4$ & $8,5,5,2,2$ \\
$5,5,4,4,1$ & $7,4,4,2,2$ \\
$5,5,4,1,1$ & $7,4,3,1,1$ \\
$5,5,1,1,1$ & $6,3,2,1,1$ \\
$5,2,1,1,1$ & $6,3,1$ \\
$2,2,1,1,1$ & $5,2$ \\
\hline
\end{tabular}
\end{table}

\section{Final remark}\label{sec4: outlook}
Recall the $q$-rising factorial $(q;q)_a:=(1-q)(1-q^2)\cdots (1-q^a)$ for $a\ge 1$ with $(q;q)_0=1$ and
 the Gaussian polynomial
\begin{equation*}
\qbin{a+b}{b}_q:= \frac{(q;q)_{a+b}}{(q;q)_{a}(q;q)_{b}}
\end{equation*}
for $a,b\ge 0$.

We also recall the $i$-th homogeneous symmetric polynomial  in $k$ variables $h_i (x_1,\ldots, x_{k})$:
\begin{equation*}
h_{i}(x_1,\ldots, x_{k})=\sum_{1\le j_1\le j_2\le \cdots \le j_i \le k} x_{j_1} x_{j_2} \cdots x_{j_i}.
\end{equation*}
Then
\begin{equation*}
\sum_{\mu\in\OO_{m\searrow}^n} z_1^{\ell_1(\mu)}\cdots z_{m-1}^{\ell_{m-1}(\mu)}q^{|\mu|}=\sum_{i=0}^{\infty} {h_i(z_1q,z_2\, q^2,\ldots, z_{m-1}q^{m-1}) }\qbin{n-1+i}{i}_{q^m}.
\end{equation*}

Setting $z_1=\cdots =z_{m-1}=z$, we get
\begin{align}\label{gf for OO_m^n}
 \sum_{\mu\in\OO_{m\searrow}^n} z^{\ell(\mu)} q^{|\mu|} &=\sum_{i=0}^{\infty}h_i (q, q^2,\ldots, q^{m-1}) \qbin{n-1+i}{i}_{q^m}\, z^i\nonumber \\
 &=\sum_{i=0}^{\infty} \left[\begin{matrix} m-2+ i \\ i  \end{matrix}\right]_q \qbin{n-1+i}{i}_{q^m}\, z^iq^i.
\end{align}

When $n \to \infty$,
\begin{align*}
\lim_{n \to \infty} \sum_{\mu\in\OO_{m\searrow}^n} z_1^{\ell_1(\mu)}\cdots z_{m-1}^{\ell_{m-1}(\mu)}q^{|\mu|}&=\sum_{i=0}^{\infty} \frac{{h_i (z_1,z_2q,\ldots, z_{m-1} q^{m-2} )} q^i}{(q^m;q^m)_i}
\end{align*}
and
\begin{align*}
\lim_{n \to \infty} \sum_{\mu\in\OO_{m\searrow}^n} z^{\ell(\mu)} q^{|\mu|}&=\sum_{i=0}^{\infty}  \qbin{m-2+i}{i} \frac{z^i q^i}{(q^m;q^m)_i},
\end{align*}
which is the generating function for $m$-falling regular partitions.

In the first of his series papers on reviving MacMahon's partition analysis, Andrews \cite{And3} gave a symbolic computational proof of \eqref{id:lecture-hall}. Is a similar approach likely to produce the generating function for the partitions enumerated by $\LL_{m\searrow}^n$, that is equivalent to \eqref{gf for OO_m^n}? If so, this may lead to an analytic proof of Theorem~\ref{thm:mfallingLH}.

\section*{Acknowledgement}
This work was initiated during the first author's visit to the Pennsylvania State University in the summer of 2018, he would like to thank the third author and the Department of Mathematics for the hospitality extended during his stay.

The first and second authors were supported by the Fundamental Research Funds for the Central Universities (No.~2018CDXYST0024) and the National Natural Science Foundation of China (No.~11501061). The third author was partially supported by a grant ($\#280903$) from the Simons Foundation.

\end{document}